\documentclass[12pt]{amsart}
\usepackage{amssymb}
\usepackage{enumerate}
\newtheorem{fact}{Fact}[section]

\newtheorem{theorem}[fact]{Theorem}
\newtheorem{definition}[fact]{Definition}
\newtheorem{exa}[fact]{Example}
\newtheorem{conj}[fact]{Conjecture}
\newtheorem{rremark}[fact]{Remark}
\newtheorem{rremarks}[fact]{Remarks}
\newtheorem{proposition}[fact]{Proposition}

\newenvironment{remark}{\begin{rremark} \rm}{\end{rremark}}
\newenvironment{remarks}{\begin{rremarks} \rm \mbox{}}{\end{rremarks}}
\newenvironment{example}{\begin{exa} \rm}{\end{exa}}
\newenvironment{conjecture}{\begin{conj} }{\end{conj}}

\newcommand{\smx}[4]{\bigl(\begin{smallmatrix}{#1}&{#2}\\{#3}&{#4}\end{smallmatrix}\bigr)}

\DeclareMathOperator{\T}{\mathbf T}
\DeclareMathOperator{\C}{\mathbf C} 
\DeclareMathOperator{\Z}{\mathbf Z}
\DeclareMathOperator{\N}{\mathbf N}
\DeclareMathOperator{\R}{\mathbf R}

\DeclareMathOperator{\Q}{\mathbf Q}
\DeclareMathOperator{\Hom}{Hom}
\DeclareMathOperator{\sing}{Sing}
\DeclareMathOperator{\lie}{Lie}
\DeclareMathOperator{\im}{Im}
\DeclareMathOperator{\Id}{Id}
\DeclareMathOperator{\codim}{codim}
\DeclareMathOperator{\diag}{diag}
\DeclareMathOperator{\sym}{Sym}
\def\iso{\cong}
\renewcommand{\phi}{\varphi}
\newcommand{\ext}{\mathrm{Ext}}
\newcommand{\gl}{\mathrm{GL}}

\title{The incidence class and the hierarchy of orbits}

\author{L. M. Feh\'er}
\address{Department of Analysis, Eotvos University Budapest, Hungary\\ R\'enyi Institute Budapest, Hungary}
\email{lfeher@renyi.hu}

\author{Zs. Patakfalvi}
\address{Department of Mathematics, University of Washington, Seattle, WA-98195, USA}
\email{pzs@math.washington.edu}

\begin{document}
\thanks{\noindent Supported by OTKA T046365MAT AMS Subject classification 32S20}

\begin{abstract} R. Rim\'anyi defined the incidence class of two singularities $\eta$ and $\xi$ as $[\eta]|_\xi$, the restriction of the Thom polynomial of $\eta$ to $\xi$. He conjectured that (under mild conditions) $[\eta]|_\xi\neq 0 \iff \xi\subset \overline\eta$. Generalizing this notion we define the incidence class of two orbits $\eta$ and $\xi$ of a representation. We give a sufficient condition (positivity) for $\xi$ to have the property that $[\eta]|_\xi\neq 0 \iff \xi\subset \overline\eta$ for any other orbit $\eta$. We show that for many interesting cases, e.g. the quiver representations of Dynkin type positivity holds for all orbits. In other words in these cases the incidence classes completely determine the hierarchy of the orbits. We also study the case of singularities where positivity doesn't hold for all orbits.
\end{abstract}

\maketitle

\section{Introduction}
Suppose that a complex algebraic representation $\rho:G\to \gl(V)$ of a complex algebraic Lie group $G$ is given and we want to understand the hierarchy of orbits: which orbit is in the closure of another one. The idea is simple: if $\eta$ and $\xi$ are two orbits then we calculate $[\eta]$---the $G$-equivariant Poincar\'e dual of $\overline\eta$ and restrict it to $\xi$. Since $[\eta]$ is supported on $\overline\eta$, if $\xi$ is disjoint from $\overline\eta$, then the incidence class  $[\eta]|_\xi$ is zero. If the $G$-action is rich enough then we have a chance for the opposite implication.

In this paper 
\begin{itemize}
\item we give a  sufficient condition (positivity) for the Incidence Property of an orbit $\xi$: for any other $G$-invariant subvariety $X\subset V$ we have that  $[X]|_\xi\neq 0 \iff \xi\subset [X]$.
\item We show that for many interesting cases, e.g. the quiver representations of Dynkin type positivity holds for all orbits.
\item  We also study the case of singularities where positivity doesn't hold for all orbits.
\end{itemize}

Our work was inspired by a conjecture of R. Rim\'anyi in \cite{rrthom}, that for singularities of holomorphic maps the incidence classes detect the hierarchy of contact singularity classes.

To study the conjecture we generalized the notion of incidence class to the general group representation setting. Let $\rho:G\to \gl(V)$ be an algebraic representation of the complex algebraic Lie group $G$ on the vector space $V$. If  $X\subset V$ is a $G$-invariant subvariety of codimension $d$ then we can assign a $G$-equivariant cohomology class $[X]\in H^{2d}_G(V)\iso H^{2d}(BG) $. In different areas of mathematics this class has different names e.g.  equivariant Poincar\'e dual, multidegree, Joseph polynomial and---in the case of singularities---Thom polynomial.

Equivariant cohomology shares some properties of ordinary cohomology. For example there is a restriction map: if $Y\subset V$ is another $G$-invariant subset and $\alpha\in H^{*}_G(V)$ then $\alpha|_Y\in H^{*}_G(Y)$. So we can define the incidence class $[X]|_Y$ measuring the ``closeness" of $X$ and $Y$.

The crucial observation is that if $\xi$ is an orbit then $[X]|_\xi$ is an equivariant Poincar\'e dual itself: Let $x\in \xi$ and let us denote by $G_x$ the maximal compact subgroup of the stabilizer group of $x$. Then  $[X]|_\xi\in H^{*}_G(\xi)\iso H^{*}_{G_x}(pt) $. Choosing a $G_x$-invariant normal space $N_x$ to $\xi$ at $x$ we will show that 

{\bf Proposition \ref{inc=tp}.} \ \ $[X]|_\xi=[X\cap N_x]_{G_x}$.

This observation reduces the problem to finding a sufficient condition for a representation having the property that all non-empty $G$-invariant subvariety has a nonzero equivariant Poincar\'e dual. This condition can be given in terms of the {\em weights} of the representation. Suppose that $T\subset G$ is a maximal torus. Then $V$ splits into 1-dimensional representations of $T$ with weights $w_i\in \check T=\Hom(T,U(1))$. If $T$ has rank $r$ then $\check T\iso \Z^r$.

{\bf Definition.} The representation $\rho:G\to \gl(V)$ is {\em positive} if the convex hull of its weights doesn't contain zero.

{\bf Theorem \ref{main}.} {\it If the representation $\rho:G\to \gl(V)$ is positive then all non-empty $G$-invariant subvariety has a nonzero equivariant Poincar\'e dual.}

From this we immediately get:

{\bf Theorem \ref{all}.} {\it All orbits of the representation $\rho:G\to \gl(V)$ have the Incidence Property if for all $x\in V$ the normal representation of $G_x$ is positive (normal positivity).} 

In section \ref{sec:examples}. we show a simple condition which implies normal positivity and show that e.g. every quiver representation of Dynkin type satisfy this condition. 

After this generalization we faced the problem, familiar to many mathematicians, that our theory doesn't apply to the original conjecture. Most orbits of the representations corresponding to the case of singularities are not positive. Many times it has a trivial reason: the stabilizer is trivial. Even if we restrict ourselves to the orbits having at least a $U(1)$ symmetry, there are plenty of non positive examples. Nonetheless, calculations show that in many cases incidence still detects the hierarchy of orbits. We list some examples, counterexamples and conjectures.

The authors thank M. Kazarian and R. Rim\'anyi for valuable discussions.

\section{The equivariant Poincar\'e dual and the Incidence class} \label{epd&inc}

We define equivariant cohomology of a $G$-space $X$ via the Borel construction and use  cohomology with rational coefficients: $H^*_G(X):=H^*(EG\times_G X;\Q)$, where $EG\to BG$ is the universal principal $G$-bundle over the classifying space $BG$ of $G$. 

We will frequently use the following simple properties of equivariant cohomology:
\begin{enumerate}
\item $H^*_G(V)\iso H^*_G(pt)$ if $V$ is contractible.
\item $H^*_G(X)\iso H^*_K(X)$, where $K$ is a maximal compact subgroup of $G$. More generally it is true if $K$ is a deformation retract of $G$.
\item We can restrict to an invariant subspace, or to a subgroup. We can also combine them. If the subspace $Y\subset X$ is invariant for a subgroup $S<G$, then we have a restriction map $H^*_G(X)\to H^*_S(Y)$.
\end{enumerate}

If $X$ is a smooth (always complex in this paper) algebraic variety (typically a complex vector space in this paper) and $Y$ is a $G$-invariant subvariety then we can assign a class $[Y]\in H^*_G(X)$. For  this class not only the names are numerous (some listed in the Introduction) but the definitions, too. These definitions are equivalent for the algebraic setting when the acting group is an algebraic torus. For an account of these results see \cite{patak}. Our definition is based on the following fact in ordinary cohomology (see  e. g.\cite{Fulton-Young-tableaux}, page 219):

\begin{proposition}[{\bf Definition}] If $X$ is a smooth algebraic variety and $Y$ is an irreducible subvariety of complex codimension $d$ then there is a unique element $[Y]\in H^*(X)$ such that
  \begin{enumerate}
  \item $[Y]$ is supported on $Y$, i. e. $[Y]$ restricted to $X\setminus Y$ is zero,
  \item $[Y]|_{X\setminus\sing Y}=[Y^o\subset (X\setminus\sing Y)]$.
  \end{enumerate}
Here $\sing Y$ denotes the singular subvariety of $Y$ and $Y^o=Y\setminus \sing Y$. The cohomology class $[Y^o\subset (X\setminus\sing Y)]$ is defined by extending the Thom class of a tubular neighbourhood of the proper submanifold $Y^o\subset (X\setminus\sing Y)$ via excision. If $Y$ has many components (of the same codimension) then we take the sum.
\end{proposition}

This definition can be extended to stratified spaces (stratified by complex submanifolds). For a variety we take the singular stratification: $Y\setminus \sing Y$, $\sing Y\setminus \sing(\sing Y)$ and so on. Now we define the the equivariant Poincar\'e dual as a $G$ characteristic class (see \cite{kaza-char}):
\begin{theorem}[{\bf Definition}]  Let $\rho:G\to \gl(V)$ be an algebraic representation and $Y\subset V$ is a $G$-invariant irreducible subvariety of complex codimension $d$, then there is a unique element $[Y]\in H^*_G(V)$ such that for all algebraic $\rho$-bundle $\pi:E\to M$ over a smooth  algebraic variety $M$ with classifying map $k:M\to BG$
       \[ [Y(E)\subset E]=\pi^*k^*[Y], \]
where $Y(E)=P\times_G Y$ for $P$ denoting the principal $G$-bundle of $E$.
  \end{theorem}
If the group action is not clear from the context we will use the notation $[Y]_G$ for the $G$-equivariant Poincar\'e dual.

Notice that the class $[Y]$ is localized at 0 in the sense that it depends only on the germ of $Y$ at 0. It follows from the fact that the restriction map $H^*_G(V)\to H^*_K(U)$ is an isomorphism for any $K$-invariant contractible neighbourhood $U$ of 0. Here we use the maximal compact subgroup $K$ to ensure the existence of small invariant contractible neighbourhoods.

Suppose now that $\xi$ is an orbit of $V$. Then we have the restriction map $|_\xi:H^*_G(V)\to H^*_G(\xi)$. Observe that $H^*_G(\xi)\iso H^*_{G_x}(pt)$, where $x\in \xi$ and $G_x$ is the (maximal compact subgroup of the) stabilizer group of $x$. We can make this more explicit if we restrict to the point $x$, which is naturally a $G_x$-space.

It is not difficult to calculate the map $H^*_G(pt)\to H^*_{G_x}(pt)$. Let $T<G$ be a maximal torus of rank $r$. Then by the Borel injectivity theorem the restriction map $H^*_G(pt)\to H^*_T(pt)$ is injective. (This is valid only with rational coefficients. That is the reason we use cohomology with rational coefficients.) So we can identify a class in $H^*_G(pt)$ with a polynomial $p\in \Q[\alpha_1,\dotsc,\alpha_r]\iso H^*_T(pt)$. We can choose $x\in \xi$ and a maximal torus $T_x<G_x$ such that $T_x<T$. Then restriction of $p$ is a linear substitution in the variables $\beta_1,\dotsc,\beta_s$, where $H^*_{T_x}(pt)\iso \Q[\beta_1,\dotsc,\beta_s]$. The substitution is determined by the inclusion $T_x\hookrightarrow T$. 

    \section{The Incidence class as an equivariant Poincar\'e dual} \label{sec:inc=epd}
Since $G_x$---the maximal compact subgroup of the stabilizer group of $x$---is a compact group, we can choose a  $G_x$-invariant normal space $N_x<T_xV$ such that $N_x\oplus T_x\xi=T_xV$. We have the exponential map $e_x:N_x\to V$ which is transversal to $\xi$. The map $e_x$ is $G_x$-invariant so it induces a homomorphism $H^*_G(V)\to H^*_{G_x}(N_x)$.

\begin{proposition} \label{inc=tp} Let $Y\subset V$ be a $G$-invariant subvariety, then $[e_x^{-1}(Y)]=e_x^*[Y]$, or, with some abuse of notation $[Y\cap N_x\subset N_x]=[Y]|_\xi$.  \end{proposition}
\begin{proof} It is enough to show that $e_x$ is transversal to $Y$ in a neighbourhood of $x$. On transversality we mean that $e_x$ is transversal to {\em every} stratum of $Y$ (see e.g. \cite[\S 2.2]{patak}). 

The action of $G$ defines a bundle map $\varphi: \lie(G)\times V\to TV$. At $x=e_x(0)$ the map $e_x$ is transversal to $\im(\phi_x)$, so transversality also holds in an open neighbourhood. But $\im(\phi_y)$ is the tangent space of the orbit of $y$, so in this neighbourhood $e_x$ is transversal to any orbit of $G$, therefore to any $G$-invariant submanifold, in particular to the strata of $Y$. 
\end{proof}

     \section{The positivity condition} \label{sec:poscond}
In this section we give a condition which implies that every non-empty $G$-invariant subvariety $Y$ has a non-zero equivariant Poincar\'e dual. 

A simple condition is if the representation $\rho: G\to \gl(V)$ ``contains the scalars" , i.e. the scalars of $\gl(V)$ are in the image of $\rho$. In this case $Y$ is automatically a cone and its projective degree can be calculated from $[Y]$ by a substitution (see e.g. \cite{dual}) so it is never 0. 

But the stabilizer group of $x\in V$ typically doesn't  contain the scalars, so we need a more general condition. By the Borel injectivity theorem we can assume that $G=\T$ is a complex torus. This is the point where we use that $G$ is a complex group, so its maximal torus is contained in a complex torus of the same rank. The real torus is too small, it always has orbits with zero equivariant Poincar\'e dual. Let us denote the set of weights by $W_\rho$.

\begin{theorem} \label{convex} For a representation $\rho:\T\to \gl(V)$ of the complex torus $\T\iso \gl(1)^r$ the following conditions are equivalent
  \begin{enumerate}
  \item for all non-empty $\T$-invariant subvariety $Y\subset V$ the class $[Y]\in H^*_{\T} (V)$ is non-zero.
  \item The convex hull of $W_\rho$ doesn't contain 0.
  \end{enumerate}
\end{theorem}
This theorem was first proved in \cite[Thm 5.2.1.]{patak}. Independently it was noticed in \cite{knutson-shimozono} that (2) $\Longrightarrow$ (1) follows from \cite[Thm D.]{knutson-miller}. Here we provide a short direct proof. We refer to condition (ii) as {\em positivity} since it means that there is a linear functional $\lambda$ on $\Z^r$ such that $\lambda(w)>0$ for all $w\in W_\rho$. another possible name would be instability as positivity is equivalent to the condition that all points of $V$ are unstable in the GIT sense.

We need one more property of the equivariant Poincar\'e dual (see \cite[prop 4.1]{patak}) to prove Theorem \ref{convex}.

\begin{theorem} \label{tpposinweights} If \,$Y\subset V$ is a $\T$-invariant subvariety then the class $[Y]\in H^*_{\T}(V)$ can be expressed as a non-zero polynomial of the weights of $\rho$ with non-negative integer coefficients.
\end{theorem}

Theorem \ref{tpposinweights} immediately follows from the fact---well known and widely used by the specialists---that for complex torus actions the notions of equivariant Poincar\'e dual and {\em multidegree} coincide. The reason we refer to \cite{patak} is that we haven't found an older proof in the literature. The proof is based on the identification of the multidegree to the equivariant intersection class \cite{MR1614555}, and then the later to the equivariant Poincar\'e dual. See \cite{four}, page 255, for the first equivalence. As for the second one, there is a cycle map from the equivariant Chow group to the equivariant cohomology ring of $V$ (\cite{MR1614555}, page 605). Scrutinizing this map, it turns out that it is in fact the ordinary cycle map of the product of some projective spaces, for which it is easy to prove that it is an isomorphism. A basic fact about multidegree is that the multidegree of a $\T$-invariant subvariety is equal to the sum of multidegrees of $\T$-invariant sub(vector)spaces (with possible positive integer multiplicities). Noticing that the  equivariant Poincar\'e dual of a subspace is the product of some weights of $\rho$ we proved Theorem \ref{tpposinweights}. Here we use the identification of a weight $w:\T\to GL(1)$ with the first Chern class $c_1(L_w)$ of the line bundle $L_w=E\T\times_w\C$.

 \begin{proof}[Proof of Theorem \ref{convex}.] We can think of the weights as linear functionals on $\R^r$---the real part of the Lie algebra of the torus $\T$. If the convex hull of $W_\rho$ doesn't contain 0 then there is a substitution $\alpha_i\mapsto z_i\in\R$ such that for all $w(\alpha_1,\dotsc,\alpha_r)\in W_\rho$ we have that $w(z_1,\dotsc,z_r)>0$. By Theorem \ref{tpposinweights} the class $[Y]$ is a non-trivial linear combination of monomials of the weights of $\rho$ with non-negative coefficients, so if we apply the same substitution we get a positive number, which implies that $[Y]=p(\alpha_1,\dotsc,\alpha_r)\neq 0$. 

On the other hand, if 0 is in the convex hull, then there is a non-trivial linear combination $\sum n_iw_i=0$ with $n_i\geq 0$ integers. It implies that $p=\prod x_i^{n_i}$ is a non-zero invariant polynomial (the coordinate $x_i$ corresponds to the one-dimensional eigenspace with weight $w_i$). Then $Y:=\{x\in V:p(x)=1\}$ is a non-empty invariant subvariety. We have that $0\not\in Y$ so by the fact that $[Y]$ is localized to 0 we get that $[Y]=0$.
 \end{proof}

 Theorem \ref{convex} immediately implies
\begin{theorem}\label{main} If the representation $\rho:G\to \gl(V)$ is positive then all non-empty $G$-invariant subvariety has a nonzero equivariant Poincar\'e dual.  
\end{theorem}

Let us recall now the definition of Incidence Property from the Introduction.
\begin{definition} \label{def:ip} Given a  representation $\rho:G\to \gl(V)$ we say that an orbit $\xi$ has the Incidence Property if for any other $G$-invariant subvariety $X\subset V$ we have that  $[X]|_\xi\neq 0 \iff \xi\subset [X]$.
\end{definition} 
Proposition \ref{inc=tp} and Theorem \ref{main} now implies that
\begin{theorem}\label{all} All orbits of the representation $\rho:G\to \gl(V)$ have the Incidence Property if for all $x\in V$ the normal representation of $G_x$ is positive (normal positivity).
\end{theorem}

\section{Examples of representations with the Incidence Property} \label{sec:examples}
First we give a stronger condition than positivity, which is easier to check. The examples we have in mind are sub-representations of the left-right action of a  subgroup of $\gl(W)\times \gl(W)$ on the vector space $\Hom(W,W)$. In these cases the weights are of the form $t-s$ where $t$ is a ``target" weight and $s$ is a ``source" weight.

 \begin{proposition} \label{i<j} Suppose that $\rho: G\to \gl(V)$ is a complex representation, $T$ is a maximal torus of $G$ and there are linearly independent elements $e_1,\dotsc,e_s\in \check T$ such that all weights of $\rho$ are of the form $e_i-e_j$ for some $i<j$. Then $\rho$ is positive, i.e. 0 is not in the convex hull of the weights of $\rho$.
 \end{proposition}
 \begin{proof}It is enough to check that 0 is not in the convex hull of $\{e_i-e_j:i<j\}$. Assume that $\sum n_{ij}(e_i-e_j)=0$ for some non-negative integers $n_{ij}$. But $\sum n_{ij}(e_i-e_j)=\sum m_k e_k$ for some integers $m_k$ with $m_1=\sum n_{1j}$. It implies that $n_{1j}=0$ for all $j$. An induction on $s$ finishes the proof.
 \end{proof}

\subsection{Schubert-varieties of flag manifolds}\label{sec:schubert}

Now we show the Incidence Property to the representation discussed in \cite{schur-schubert} and \cite{knutson-miller}.    

Let the group $B^+\times B^-$ act on $\Hom(\C^n,\C^n) \cong \{(m_{ij})_{i,j=1}^{n}\}$, where $B^+$ and $B^-$ are the groups of $n \times n$ upper and lower triangular matrices, respectively, and $(R,L) \cdot M = LM R^{-1}$ for $R \in B^+$, $L \in B^-$ and $M \in \Hom(\C^n,\C^n)$.

Every orbit contains a unique 0-1 matrix $A$, every column and row of which contains at most one 1. (Note that the full rank matrices correspond to the Schubert varieties of the $n$-dimensional flag-manifold.) We can expand such a matrix uniquely to the right and down to a $2n \times 2n$ permutation matrix. This expansion is obtained by adding 1's to rows not containing 1's starting from the top, putting the 1 in the left-most position, where it is possible, outside of the upper-left $n \times n$ matrix. This way we can associate a permutation $\pi \in S_{2n}$ to an orbit, by setting $\pi(i)=j$, if there is a 1 in the $(i,j)$ position of the permutation matrix obtained.

According to \cite[\S4.]{schur-schubert} the maximal torus of the stabilizer group of the matrix $A$ with expanded permutation $\pi$ is
\[ T_{\pi}=\{\left(\diag(\alpha_1,\dotsc,\alpha_n),\diag(\alpha_{\pi(1)},\dotsc,\alpha_{\pi(n)})\right) | \alpha_i \in U(1) \} \]
and an invariant normal space to the orbit of $A$ is
\[ N_A= \{(m_{ij})_{i,j=1}^n | \textrm{ $m_{ij}=0$ if $ \pi(i)\leq j$ or $\pi^{-1}(j)\leq i$}\} \]
of $\Hom(\mathbb{C}^n, \mathbb{C}^{n})$ invariant under $G_{\pi}$ action. 

Using the basis ${e_1,\dotsc,e_n}$ for $\check T_\pi\iso U(1)^n$ the weights of $N_A$ are of the form $e_{\pi(i)}-e_j$ where  $\pi(i)> j\}$, so by Proposition \ref{i<j}. (for the reverse ordering) this orbit has the Incidence Property.

Equivalent statements were proved in \cite{las-sch}, \cite{kumar}, \cite{goldin} and \cite{rim-buch} as a characterization of the {\em Bruhat order}. In fact a formula in \cite{rim-buch} gave us the idea to study the relation of incidence with multidegree.

\subsection{Quiver representations} \label{sec:quiver}

In this section we show that for representations corresponding to quivers of Dynkin type all orbits have the Incidence Property. The equivariant Poincar\'e duals for these cases were first calculated in \cite{quiver}. It also contains a more detailed introduction to the geometry of these representations. Recently another algorithm was given in \cite{knutson-shimozono}.

Consider an oriented graph $Q$ and denote by $Q_0$ the set of its vertices and by $Q_1$ the set of its edges. If $e$ is an edge of $Q$, then let $e'$ and $e''$ denote the  head and the tail of $e$, respectively.

If a function $d:Q_0 \to\mathbb{N}$  (dimension vector) is given we have the group $G=\displaystyle\bigoplus_{i \in Q_0} \gl (d(i))$ acting on $V=\displaystyle\bigoplus_{e \in Q_1} \Hom(\mathbb{C}^{d(e')}, \mathbb{C}^{d(e'')})$, by the formula
\begin{equation*}
\left(\bigoplus_{i \in Q_0} A_i \right) \cdot \left( \bigoplus_{e \in Q_1} \phi_e \right) = \bigoplus_{e \in Q_1}\left( A_{e''} \phi_e A_{e'}^{-1} \right)
\end{equation*}
 The orbits of this representation correspond to the representations (modules) of the path algebra $\mathbb{C}Q$, with dimension vector $d$. 

For graphs of Dynkin type, the set  $R(Q)$ of indecomposable modules of $\mathbb{C}Q$ is finite. Any $\mathbb{C}Q$-module can be decomposed into a  form of $\sum_{r \in R(Q)} \mu_r r$ for some integer numbers $\mu_r$ by the Krull-Schmidt theorem and the numbers $\mu_r$ are well defined. The maximal compact stabilizer subgroup of $M$ is $G_M = \displaystyle\bigoplus_{r \in R(Q)} U( \mu_r)$. A normal space is $N_M = \displaystyle\bigoplus_{r,s \in R(Q)} \Hom(\mathbb{C}^{\mu_r}, \mathbb{C}^{\mu_s})^{m_{rs}}$, where $m_{rs}= \dim \ext_{\mathbb{C}Q}(r, s)$, and $G_{M}$ acts on $N_M$ with the rule 
\begin{equation*}
\left(\bigoplus_{r \in R(Q)} A_r \right) \cdot \left( \bigoplus_{r,s \in R(Q)} \bigoplus_{i=1}^{m_{rs}} \phi_{rs}^i \right) = \left( \bigoplus_{r,s \in R(Q)} \bigoplus_{i=1}^{m_{rs}} A_s\phi_{rs}^i A_r^{-1} \right)
\end{equation*}
where $\phi_{rs}^i\in \Hom(\mathbb{C}^{\mu_r}, \mathbb{C}^{\mu_s})$ for $i=1,\dotsc,m_{rs}$.

The maximal torus $T_M$ of the stabilizer group of $M$ is isomorphic to $U(1)^{\sum\mu_r}$. Let us denote the standard basis of the weight lattice $\check T_M$ by $\{e_{r,j}: r \in R(Q),\ j\leq \mu_r\}$. Then the weights of the representation of $G_M$ on $\Hom(\mathbb{C}^{\mu_r}, \mathbb{C}^{\mu_s})$ are $e_{s,i}-e_{r,j}$ for any $1 \leq i \leq \mu_s$ and $1 \leq j \leq \mu_r$.  To apply Proposition \ref{i<j} it is enough to show that there is an ordering $\succ$ on $R(Q)$, such that if $m_{rs}\neq 0$, then $r \succ s$. 

We recall the notion of the {\em Auslander-Reiten translate} $\tau$ which is partial self mapping map of $ R(Q)$ with the following properties.
\begin{enumerate}
\item \label{preproj} For every $r \in R(Q)$, there is a unique $n_r \in \mathbb{N}$, such that $\tau^{n_r} r$ is projective. (In other words: in our case all indecomposables are pre-projective.)
\item $\ext_{\mathbb{C}Q}(r, s)=\ext_{\mathbb{C}Q}(\tau r, \tau s)$ for every $r,s \in R(Q)$, where both $\tau r$ and $\tau s$ are defined.
\end{enumerate}

(\ref{preproj}) implies that there is an ordering $\succ$ on $R(Q)$, for which if $n_r>n_s$, then $r \succ s$. Suppose that $r \not\succ s$. Then $n_r \leq n_s$, and consequently $m_{rs}=\dim \ext_{\mathbb{C}Q}(l_r, l_s) = \dim \ext_{\mathbb{C}Q}(\tau^{n_r}l_r, \tau^{n_r}l_s) = 0$, since $\tau^{n_r} l_r$ is projective. This shows that indeed if $m_{rs}\neq 0$, then $r \succ s$, and we proved

\begin{theorem}\label{quiver} For a representation corresponding to a Dynkin type quiver all orbits satisfy the Incidence Property. \end{theorem}

%
\section{Singularities} \label{sec:singularities}

In this section we study incidences of singularities of holomorphic map germs. On singularity we mean a contact orbit of map germs. The equivariant Poincar\'e dual in this context is called {\em Thom polynomial}. We use \cite{rrthom} and \cite{tpforgroup} as a general reference.

The space $\mathcal{E}(n,p)$ of holomorphic map germs from $\C^n$ to $\C^p$ and the contact group $\mathcal{K}(n,p)$ acting on it are infinite dimensional. To define $G$-equivariant cohomology for infinite dimensional groups some extra care is needed as it is explained in \cite{rrthom}. Our main interest lies in studying {\em finite codimensional} singularities where a finite dimensional reduction is available.

Only finite codimensional singularities have Thom polynomials so it is a natural choice. It is possible to restrict to infinite codimensional singularities (for example see \cite{integer}), but to simplify the situation we stick with the finite codimensional case.

Every finite codimensional singularity $\eta$ is {\em finitely determined}, i.e. there is a $k\in \N$ such that for any two germs $f\in \eta$ and $g\in \mathcal{E}(n,p)$ we have that $g\in \eta$ if and only if their  $k$\textsuperscript{th} Taylor polynomials (or $k$-jets) $j^k(f)$ and $j^k(g)$ are contact equivalent. So we can reduce to the finite dimensional space of $k$-jets: $J^k(n,p)\iso \bigoplus_{i=1}^k\Hom(\sym^i\C^n,\C^p)$.  We have the truncating map $t_k:\mathcal{E}(n,p)\to J^k(n,p)$. Similarly we can define $\mathcal{K}^k(n,p)$, the group of $k$-jets of the contact group:
\[ \mathcal{K}^k(n,p):=\{(\varphi,\psi)\in J^k(n,n)^\times\times J^k(n+p,n+p)^\times:\psi|_{\C^n\times0}=\Id_{\C^n},\ \pi_{\C^n}\psi=\psi\pi_{\C^n}\}. \]
This group acts on $J^k(n,p)$ and for the truncating homomorphism $\pi_k: \mathcal{K}(n,p)\to \mathcal{K}^k(n,p)$ the map $t_k$ is $\pi_k$-equivariant. The homomorphism $\pi_k$ induces isomorphism on equivariant cohomology and for any two $k$-determined singularities $\eta$ and $\xi$:
\[ [\eta]|_\xi=[t_k\eta]|_{t_k\xi}. \]
We don't prove this statement---as we avoided defining the left hand side---but use as a motivation to study the representation of the algebraic group $\mathcal{K}^k(n,p)$ on the jet space $J^k(n,p)$. 

We cannot expect that all orbits of this representation have the Incidence Property. There are many singularities $\xi$ with small symmetry: the maximal torus of their stabilizer group is trivial. It implies that the restriction map is automatically trivial. To have a chance for the Incidence Property to be satisfied we have to require the existence of at least a $U(1)$ symmetry. Since the diagonal torus $U(1)^n\times U(1)^p$ is maximal in $\mathcal{K}^k(n,p)$ and all maximal tori are conjugate, this requirement implies that a representative $f=(f_1,\dotsc,f_p)$ of the orbit $\xi$ with diagonal symmetry can be chosen. Then $f$ is a {\em weighted homogeneous} polynomial, i.e. there are weights $a_1,\dotsc,a_n,b_1,\dotsc,b_p\in \Z$ such that
\[ f_j(t^{a_1}x_1,\dotsc,t^{a_n}x_n)=t^{b_j}f_j(x_1,\dotsc,x_n) \]
for $j=1,\dotsc,p$. These polynomials are also called quasi-homogeneous.

The Incidence Property can fail even for weighted homogeneous polynomials, which can be detected by restricting them to themselves. According to the definition. Namely if $\xi$ has the Incidence Property then $e(\xi)=[\xi]|_\xi\neq0$. 

\begin{example}\label{exa:e=0} For $\xi=(x^2+3 yz, y^2+3 xz, z^2+3 xy)$ the normal Euler class $e(\xi)=0$. \end{example}
\begin{remark}\label{family}
  This Euler class can be directly calculated using unfolding---see the second part of this section---but there is a deeper reason for the vanishing of $e(\xi)$. This orbit is a member of a one-parameter family of orbits: $f_\lambda=(x^2+\lambda yz, y^2+\lambda xz, z^2+\lambda xy)$. (This is a famous example (\cite{MR0432666}), the smallest codimensional occurence of a family in the equidimensional case.) The tangent direction of the family in the normal space $N_\xi$ has 0 weight, which implies that $e(\xi)=0$. In fact we don't know any example when $e(\xi)=0$ but $\xi$ is {\em not in a family}.
\end{remark}
The following conjecture is a slightly modified version of Rim\'anyi's from \cite{rrthom}:
\begin{conjecture}The singularity $\xi\subset J^k(n,p)$ has the Incidence Property iff $e(\xi)\neq0$.\end{conjecture}

\begin{remark} Defining the Incidence Property (Definition \ref{def:ip}) in such a way that we require the condition $e(\xi)=[\xi]|_\xi\neq0$ to be satisfied seems formal, we know that $\xi$ is a subset of $\overline\xi$ anyway. However the vanishing of $e(\xi)$ sometimes can be related to the vanishing of a proper incidence $[\eta]|_\xi$ where $\eta\neq\xi$ but $\xi\subset \overline\eta$. Consider the following abstract example:
\end{remark}

\begin{example} Let $GL(1)$ act on $\C^2$ by $z(x,y)=(x,zy)$. The orbits of this representation are $f_\lambda=\{(\lambda,0)\}$ and $g_\lambda=\{(\lambda,y):\, y\neq0\}$. The fixed point $f_\lambda$ has a $U(1)$ symmetry and $f_\lambda\subset\overline{g_\lambda}$ but the incidence $[g_\lambda]|_{f_\lambda}$ is zero. We have a geometric reason for this: $[g_\mu]|_{f_\lambda}=0$ if $\mu\neq\lambda$ since $f_\lambda\not\subset\overline{g_\mu}$. But $[g_\mu]$ depends continuously on $\mu$ so it is constant. We can also calculate directly: $\overline{g_\lambda}$ is a line so $[g_\lambda]|_{f_\lambda}$ is the Euler class of the complementary invariant line $\{(\lambda,0):\lambda\in \C\}$. But this is exactly the set of fixed points so the Euler class is zero.

We believe that such situation is not uncommon for singularities, i.e. there are contact orbits $f_\lambda$ and $g_\lambda$ depending on the parameter $\lambda$ continuously such that $\codim(f_\lambda)=\codim(g_\lambda)+1$, $f_\lambda\subset\overline{g_\lambda}$ and $f_\lambda$ has a $U(1)$ symmetry. The same reasoning shows that the incidence $[g_\lambda]|_{f_\lambda}$ is zero in this case. Unfortunately we were unable to verify the existence of such an example.
\end{example}

Theorem \ref{main} implies that
\begin{proposition} If a singularity $\xi\subset J^k(n,p)$ is positive i.e the normal action of the stabilizer group is positive, then $\xi$ has the Incidence Property.
\end{proposition}
  
Complicated singularities are usually not positive:
\begin{example} \label{exa:xayb} The singularities $(x^a,y^b)$ for $2\leq a\leq b$ are positive exactly for $(x^2,y^2)$,  $(x^2,y^3)$ and  $(x^2,y^4)$. \end{example}

On the other hand we know many positive cases:
\begin{example}\label{exa:s2} All $\Sigma^1$ and $\Sigma^{2,0}$ singularities (type $A_n,\ I_{a,b},\ III_{a,b},\ IV_a,\ V_a$, for the notation see \cite{mather}) are positive.
\end{example}
\begin{remark} In fact all the incidences of the singularities in Example \ref{exa:s2} can be explicitly calculated since their adjacencies are well understood (see \cite{rrthesis}). For example in the equidimensional case for $2\leq a\leq b$, $(a,b)\neq(2,2)$, $2\leq c\leq d$ and $c+d<a+b$ we have \cite{rrprivate}:
\[  [I_{c,d}]|_{I_{a,b}}=(a-1)!(b-1)!\left(\frac{a^db^c}{(a-c)!(b-d)!}+\frac{a^cb^d}{(a-d)!(b-c)!}\right)\cdot\frac1{\delta_{cd}+1}\cdot g,\]
where $k!=\infty$ for $k$ negative and $\delta_{cd}=1$ if $c=d$ and $0$ otherwise. The maximal torus of the stabilizer group of $I_{a,b}$ for $(a,b)\neq(2,2)$ is $U(1)$ and $g$ denotes the generator of $\Z\iso H^{D}_{U(1)}(pt;\Z)$ where $D$ is the degree of the cohomology class $[I_{c,d}]$.
\end{remark}

We have one result which goes beyond the positivity condition. To state it, we recall the definition of $\Sigma^i$ classes:
\begin{definition}
  \[ \Sigma^i:=\Sigma^i(n,p)=\{f\in J^k(n,p):\dim\ker d_0f=i\}.\]
\end{definition}
We supressed the dependence on $k$ in the notation as the $\Sigma^i$-class of $f$ depends only on the linear term $j^1f$.
\begin{theorem}\label{sigma2} Suppose that the contact orbit $\xi\subset \Sigma^0\cup\Sigma^1\cup\Sigma^2$ has at least a $U(1)$ symmetry (i.e. there is a weighted homogeneous polynomial in the orbit). Then the incidence $[\Sigma^2]|_\xi$ vanishes if and only if $\xi\not\in \Sigma^2$.
\end{theorem}

Notice that the classification of $\Sigma^2$ singularities is not known, so case by case checking is not available here. To prepare the proof we make some remarks.
\begin{remarks}

  \begin{enumerate}[(i)]
\item If  the contact orbit $\xi$ is in $ \Sigma^0\cup\Sigma^1\cup\Sigma^2$ then $\xi$ has a representative $f$ such that $f$ depends only on the first two variables so we will assume that $n=2$. If $f\in J^k(2,p)$ then $f\in \Sigma^2$ if and only if $f$ has no linear terms.
  \item $\Sigma^2$ always contains a contact orbit which is open in $\Sigma^2$. If $k=1$ then $\Sigma^2$ is an orbit. If $k\geq 2$ and $p=2$ then $I_{2,2}$---the orbit of $(x^2,y^2)$---is open in $\Sigma^2$. If $p>2$ then $III_{2,2}$---the orbit of $(x^2,xy,y^2,0,\dotsc,0)$---is open in $\Sigma^2$. So for example for the latter case $[III_{2,2}]=[\Sigma^2]$ (see \cite{mather}).
  \item By the Thom-Porteous-Giambelli formula \cite{porteous} we have
\[ [\Sigma^{2}(2,p)]=\prod(\beta_{i}-\alpha_{1})(\beta_{i}-\alpha_{2})\]
 in terms of Chern roots. In other words $\alpha_1,\alpha_2,\beta_1,\dotsc,\beta_p$ denote the generators of $H^2_T(pt)$ where $T=U(1)^2\times U(1)^p$ is the maximal torus acting on $J^1(2,p)$. (It can also be seen directly since $J^1(2,p)\cap \Sigma^{2}(2,p)=0$.)
  \item Consequently for the weighted homogeneous polynomial $f\in J^k(2,p)$ with one $U(1)$ symmetry, i.e. with stabilizer subgroup of rank one, and with weights $ a_1,a_2$ and $b_1,\ldots,b_p$ we have 
\begin{equation}[\Sigma^2]|_f=\prod(b_i-a_1)(b_i-a_2)g^{2p}\label{resu} \end{equation} 
 where $ g $ is the generator of the cohomology ring $ H^2_{U(1)}(pt;\Z)\iso \Z[g] $ of the symmetry group.  In other words the incidence is zero if and only if there is a target weight equal to a source weight. 

If $f$ has symmetry group $U(1) \times U(1)$, then we can still restrict $[\Sigma^2]|_f$ further to the subtorus $U$ of $U(1) \times U(1)$ corresponding to the weighting of $f$. Then we get formula (\ref{resu}) for  $\left( [\Sigma^2] |_f \right) |_U$. Hence if no target weight is equal to a source weight we still get that the incidence is not zero.
  \item Notice that (\ref{resu}) holds only if $f$ has no additional symmetry. For example if $f$ is contact equivalent to a monomial germ then the rank of the symmetry group is at least 2. If $f\in \Sigma^2(2,p)$ is monomial then the incidence $ [\Sigma^2]|_f $ is not zero since we can restrict to the $U(1)$ symmetry corresponding to the source weights $ a_1=a_2=1 $ and see that $b_i\geq 2$ for all $i\leq p$ since $f\in \Sigma^2(2,p)$ implies that there are no linear terms. 
  \end{enumerate}
\end{remarks}

These remarks imply that to prove Theorem \ref{sigma2} it is enough to prove:
\begin{proposition} Let $f\in J^k(2,p)$ be a weighted homogeneous polynomial with weights $ a_1,a_2$ and $b_1,\ldots,b_p$. If $a_1=b_1$ then $f$ is contact equivalent to a monomial germ.
 \end{proposition}
\begin{proof}
We distinguish three cases and in all cases we will use the fact that if the ideals $ (f_1,\dotsc,f_p) $ and $ (f^{'}_1,\dotsc,f^{'}_p) $ agree then $f$ is contact equivalent to $f^{'}$\negthinspace. In the first two cases we don't need the $a_1=b_1$ assumption.

{\bf Case 1:} If $a_1=0$ or $a_2=0$.  If, say, $a_2=0$ then $f_i=x^{b_i/a_1}g_i(y)$ for some polynomials $g_i(y)$. If the lowest power of $y$ in $g_i(y)$ is $k_i$ then $g_i(y)=y^{k_i}(h_i+y\overline g_i(y)$ for some polynomials $\overline g_i(y)$ and non-zero numbers $h_i$. The functions $h_i+y\overline g_i(y)$ are units in the algebra $\C[y]/(y^{k+1})$ which implies that the ideal $ (f_1,\dotsc,f_p) $ is equal to $ (x^{b_1/a_1}y^{k_1},\dotsc,x^{b_p/a_1}y^{k_p}) $ i.e. $f$ is contact equivalent to a monomial germ.

{\bf Case 2:} If $a_1>0$ and $a_2<0$. Suppose that $\deg x^uy^v=\deg x^{u'}y^{v'}$ and $u\geq u'$. Then $v\geq v'$.It implies that $f_i(x,y)=x^{u_i}y^{v_i}(h_i+g_i(x,y))$ for some $u_i,\ v_i$ non-negative integers and $h_i\neq 0$ constant. Therefore $ (f_1,\dotsc,f_p)=(x^{u_1}y^{v_1},\dotsc,x^{u_p}y^{v_p})$.

{\bf Case 3:} If $a_1>0$ and $a_2>0$. Then $a_1=b_1$ implies that $f_1=y^{k}$ for $a_1=ka_2$. Without loss of generality we can assume that $a_1=k$ and $a_2=1$. From here we don't use the $a_1=b_1$ assumption. There is a unique decomposition $b_i=u_ik+d_i$ for $0 \leq d_i < k$. Then $(f_i,y^k)=(x^{u_i}y^{d_i},y^k)$ for $2\leq i\leq p$, therefore $ (f_1,\dotsc,f_p)=(y^k,x^{u_2}y^{d_2},\dotsc,x^{u_p}y^{d_p})$.
\end{proof}

\subsection{Calculations of the Examples \ref{exa:xayb} and \ref{exa:s2}} \label{sec:calc}

Let us recall now that the normal space of the contact orbit of the singularity $\xi=(f_1,\dotsc,f_p)$ is isomorphic
to $\C^n\oplus U_\xi$ where $U_\xi=J^k(n,p)/(f_ie_j,\partial f/x_l:i,j\leq p,l\leq n)$ is the {\em unfolding space} of $\xi$. Here $e_j$ denotes the constant map $(0,\dotsc,0,1,0,\dotsc,0)$ with 1 in the $j$\textsuperscript{th} coordinate. (see e.g. in \cite{agv} or \cite{tpforgroup}).

{\bf \large $(x^a,y^b)$:} The unfolding space is spanned by the monomials $\{(x^iy^j,0): i<a-1, j<b\}$ and $\{(0,x^iy^j): i<a,j<b-1\}$ (except for $i=0,\ j=0$---constants are not in $J^k(n,p)$). The maximal torus of the symmetry is a $U(1)^2=\{(\alpha,\beta):\alpha,\beta\in U(1)\}$ acting via $\bigl(\smx\alpha00\beta,\smx{\alpha^a}00{\beta^b}\bigr)$ on $\Hom(\C^2,\C^2)$ therefore the weights of the normal space are 
$$\{(i,j):i=2,\dotsc,a;\, j=-b+1,\dotsc,0\}\setminus \{(a,0)\}$$ 
and 
$$\{(i,j):i=-a+1,\dotsc,0;\, j=2,\dotsc,b\}\setminus \{(0,b)\}.$$ By sketching the distribution of the weights on the coordinate plane we can see that $(0,0)$ is in the convex hull unless $a=2$ and $b=2,3$ or $4$.

To demonstrate that $\Sigma^1$ and $\Sigma^{2,0}$ singularities are positive we pick the series $III_{a,b}$. the other calculations are similar but simpler since the symmetry is $U(1)$.

{\bf \large $III_{a,b}=(x^a,xy,y^b)$:} The unfolding space is spanned by 
$$\{(x^j,0,0):j=1,\dotsc,a-2\}\cup \{(y^j,0,0):j=1,\dotsc,b-1\}\cup $$
$$\{(0,0,x^j):j=1,\dotsc,a-1\}\cup\{(0,0,y^j):j=1,\dotsc,b-2\}\cup \{(ax^{a-1},y,0),(0,x,by^{b-1})\}.$$
The symmetry is $U(1)^2$ and the weights are
\[ \{(1,0),(2,0),\dotsc,(a-1,0);(a,-1),(a,-2),\dotsc,(a,1-b);(0,1),\dotsc,(0,b-1);(-1,b),\dotsc,(1-a,b)\}.\]
The linear functional $bx+ay$ is positive on all these weights  proving the positivity of $III_{a,b}$.

\bibliography{inci}

\bibliographystyle{alpha}

\end{document}